\documentclass[12pt,reqno]{amsart}
\usepackage[unicode,bookmarks,colorlinks]{hyperref}

\usepackage{amsmath,amssymb,amsthm}
\usepackage{amsfonts}
\usepackage[mathscr]{eucal}
\usepackage{mathrsfs}
\usepackage{graphicx}
\usepackage{color}
\usepackage[applemac]{inputenc}
\usepackage{fontenc}
\usepackage{amsmath,amssymb,amsthm}
\usepackage{amsfonts}
\usepackage[mathscr]{eucal}
\newtheorem{theorem}{Theorem}[section]
\newtheorem{lemma}[theorem]{Lemma}

\theoremstyle{definition}

\newtheorem{example}[theorem]{Example}

\theoremstyle{remark}
\newtheorem{remark}[theorem]{Remark}

\numberwithin{equation}{section}

\def\a{\alpha}

\def\P{{\mathbb P}}

\begin{document}

\title{\bf Correlation structure of time-changed fractional Brownian motion}

\author{Jebessa B. Mijena}
\address{Jebessa B. Mijena, 231 W. Hancock St, Campus Box 17, Department of Mathematics,
Georgia College \& State University, Milledgeville, GA 31061}
\email{jebessa.mijena@gcsu.edu}

\begin{abstract}
Fractional Brownian motion (fBm) is a centered self-similar Gaussian process with stationary increments, which depends on a parameter $H \in (0, 1)$ called the Hurst index. The use of time-changed processes in modeling often requires the knowledge of their second order properties such as covariance function. This paper provides the explicit expression for the correlation structure for time-changed fractional Brownian motion. 
 Several examples useful in applications are discussed.

\end{abstract}


\maketitle

\section{Introduction}
One of the most important stochastic processes used in a variety of applications is the {\it Brownian motion} or {\it Wiener process} $B = \{B(t), t\geq 0\},$ which is a Gaussian process with zero mean and covariance function $\mbox{min}(s,t).$ The process $B$ has independent increments. Scientist used to model natural phenomena such as rainfall, river levels, temperature, as simple random walk processes or Brownian motion processes. However, during designing an optimal dam for the river Nile when hydrologist Hurst tried to model the river levels over the year as a Brownian motion process, he discovered to his surprise that the river level is not totally random. Instead the process increments have some vivid correlation, which indicates that the natural phenomena of river level fluctuation follows a biased random walk or fractional Brownian motion path more than that of a regular Brownian motion. Later, re-scaled range analysis revealed that this fact is true for several other natural processes including lake levels, rainfall, temperature, sunspot counts, and tree rings etc \cite{feder}. Researchers have applied fractional Brownian motion to a wide range of problems, such as particle diffusion, DNA sequences, bacterial colonies, geophysical data, electrochemical deposition, and stock market indicators \cite{addison}. In particular, computer science applications of fractional Brownian motion include modeling network traffic and generating graphical landscapes \cite{caglar}, \cite{mandelbort1}.

A Gaussian process $B_H = \{B_H(t), t\geq 0\}$ is called \textit{fractional Brownian motion} (fBm) of Hurst parameter $H\in (0,1)$ if it has mean zero and the covariance function
\begin{equation}\label{covfBm}
\mathbb{E}[B_H(t)B_H(s)] = \frac{\sigma^2}{2}(s^{2H} + t^{2H} - |t-s|^{2H}),
\end{equation}
where $\sigma^2 = \mbox{Var}(B_H(1)).$
This process was introduced by Kolmogorov \cite{kolmogorov-1} and studied by Mandelbrot and Van Ness in \cite{mandelbrot}, where a stochastic integral representation in terms of a standard Brownian motion was established. The parameter $H$ is called Hurst index from the statistical analysis, developed by the climatologist Hurst \cite{hurst}, of the yearly water run-offs of Nile river.\\

The fractional Brownian motion has the following properties.\\

1. {\it Self-similarity:} For any constant $a>0,$ the processes $\{a^{-H}B_H(at), t\geq 0\}$ and $\{B_H(t), t\geq 0\}$ have the same probability distribution. This property is an immediate consequence of the fact the covariance function \eqref{covfBm} is homogeneous of order $2H$, and it can be considered as a "fractal property" in probability.\\

2. {\it Stationary increments:} From \eqref{covfBm} it follows that the increment of the process in an interval $[s,t]$ has a normal distribution with mean zero and variance 
$$\mathbb{E}[(B_H(t) - B_H(s))^2] =\sigma^2 |t-s|^{2H}.$$
Hence, for any integer $k\geq 1$ we have
$$\mathbb{E}[(B_H(t)-B_H(s))^{2k}] = \frac{\sigma^2(2k)!}{k!2^k}|t-s|^{2Hk}.$$

For $H=1/2$, the covariance function can be written as $R_{1/2}(t, s) = \mbox{min}(s, t)$, and the process $B_{1/2}$ is an ordinary Brownian motion. In this case the increments of the process in disjoint intervals are independent. However, for $H\neq 1/2,$ the increments are not independent.

Since fractional Brownian motion $\{B_H(t), t \geq 0 \}$ has stationary increments, its increments $$X_j = B_H(j) - B_H(j-1), \ j = 1, \cdots,$$
form a stationary sequence. The sequence $\{X_j, j = 1, \cdots\}$ is called {\it fractional Gaussian noise} (FGN). It is a Gaussian stationary sequence with covariance function
\begin{eqnarray}
r(j) &=& \frac{\sigma_0^2}{2}\left(|j+1|^{2H} - 2|j|^{2H} + |j-1|^{2H}\right)\nonumber\\
&\sim& \sigma^2_0H(2H-1)j^{2H-2}, \ \ \ \mbox{as}\ \ \ j\rightarrow\infty,\nonumber
\end{eqnarray}
where $\sigma^2_0 = \mbox{Var}(X_j).$ $r(j)$ tends to $0$ as $j\rightarrow\infty$ for all $0<H<1$, when $1/2 < H < 1$ it tends to zero so slowly that $\sum_{j=1}^{\infty} r(j) = \infty$ diverges. We say that the sequence $\{X_j, j \in \mathbb{Z}\}$ exhibits {\it long-range dependence}. Moreover, this sequence presents an aggregation behavior which can be used to describe cluster phenomena. For $0< H < 1/2$, $\sum_{j=1}^{\infty}|r(j)| < \infty$ and $\sum_{j=1}^{\infty}r(j) = -\sigma^2_0/2$ \cite[Prop. 7.2.10]{samorodnitsk-taqqu}. Although there is no long-range dependence, the case $0 < H < 1/2$ is a singular one. Because the coefficient $H(2H-1)$ is negative, the $r(j)$ are negative for all $j,$ a behavior sometimes referred to as "negativee dependence." 

In this paper, we consider a class of intreated self-similar processes formed by subordinated fractional Brownian motion. Let $\{Y: Y(t), t\geq 0\}$ be the inverse subordinator (see section \eqref{moments} for its definition). Let $Z = \{Z(t), t\geq 0\}$ be a stochastic process defined as $Z(t) = B_H(Y(t)),\ t\geq 0.$ We call this intreated process {\it subordinated fractional Brownian motion}.

Since the sample paths of $B_H$ and $Y$ are continuous, the subordinated fBm also has continuous sample paths. However, $Z$ is non-Markovian, non-Gaussian and does not have stationary or independent increments. When $H=1/2$, we call the process $Z(t) = B_{1/2}(Y(t))$ subordinated Brownian motion. Its path properties were investigated in Magdziarz \cite{magdziarz}; Nane \cite{nane}.

The present paper computes the correlation function of $Z(t) = B_H(Y(t)),$ where the inner process is any inverse subordinator. Then in Section \eqref{application} the explicit formula is derived for the correlation function of time-changed processes that rise in applications.

\section{Moments of Increments}\label{moments}
Consider a non-decreasing L$\acute{\mbox{e}}$vy process $\{D(s), s\geq 0\},$ starting from $0$, which is continuous from the right with left limits. Such a process is called a subordinator. It has stationary and independent increments and is characterized by its Laplace transform
$$\mathbb{E}e^{-\lambda D(s)} = e^{-s\phi(\lambda)}, \lambda\geq 0.$$
where the Laplace exponent $\phi$ is a Bernstein function given by
\begin{equation}\label{laplaceexponent}
\phi(\lambda) = \mu\lambda + \int_{(0,\infty)}(1 - e^{-\lambda x})v(dx), \lambda\geq 0.
\end{equation}
where $\mu\geq 0$ is the drift and $v$ is a L$\acute{\mbox{e}}$vy measure on $\mathbb{R}^+\cup \{0\}$ which satisfies $\int_{0}^{\infty}(1\wedge x)v(d\,x) < \infty $\ (see \cite{Applebaum}). If the drift coefficient $\mu = 0,$ or if the L$\acute{\mbox{e}}$vy measure $v$ satisfies $v(0,\infty)=\infty,$ then $D$ is strictly increasing.

The first-passage time of a subordinator $\{D(s), s\geq 0\}$, is a new process $\{Y(t), t\geq 0\},$ called an \textit{inverse subordinator}, and is defined as follows:
\begin{equation}\label{inversesubordinator}
Y(t) = \mbox{inf}\{s: D(s) > t\},\ \ t\geq 0.
\end{equation}
We have $\{Y(t) < u\} = \{D(u) > t\}$ \cite[Eq.(3.2)]{mark2}.

In this section we compute moments of an increments of general inverse subordinator.

Before we calculate moments, we first argue that all moments of an inverse subordinator are finite. Notice that, for any $x > 0,$ we can bound the tail distribution of $Y$ using Markov's inequality:
\begin{equation}\label{taildistribution}
P[Y(t) > s]\leq P[D(s)\leq t] = P[e^{-xD(s)}\geq e^{-xt}]\leq e^{xt}\mathbb{E}e^{-xD(s)} = e^{xt-s\phi(x)},
\end{equation}
which implies $\mathbb{E}[Y(t)^\alpha] < \infty$ for any $\alpha > 0.$

Let $U^\gamma(t) = \mathbb{E}[Y(t)^\gamma]$. Using \cite[Eq. 3.13]{mark}                        
  the Laplace transform of $U^\gamma(t)$ with $\gamma > -1$ is given by:
\begin{eqnarray}\label{laplacetranformofinverse}
\widetilde{U^\gamma}(\lambda) &=& \int_{0}^{\infty}e^{-\lambda t}\int_{0}^{\infty}s^{\gamma}f_{Y(t)}(s)\, ds\, dt\nonumber\\
&=& \int_{0}^{\infty}s^{\gamma}\int_{0}^{\infty}e^{-\lambda t}f_{Y(t)}(s)\, dt\, ds\nonumber\\
&=&\frac{\phi(\lambda)}{\lambda}\int_{0}^{\infty}s^{\gamma}e^{-s\phi(\lambda)}\, ds\nonumber\\
&=&\frac{\Gamma(1 + \gamma)}{\lambda[\phi(\lambda)]^\gamma},
\end{eqnarray}
where
$f_{Y(t)}(s)$ is the probability density of   inverse subordinator.
Of particular importance
is the mean of $Y(t)$. 
From \eqref{laplacetranformofinverse}, $U$ has Laplace transform given by
\begin{equation}\label{laplaceofrenewalfunction}
\widetilde{U}(\lambda) = \frac{1}{\lambda\phi(\lambda)}.
\end{equation}
Thus, $U$ characterizes the inverse process $Y$, since $\phi$ characterizes $D$. For example, using \eqref{laplacetranformofinverse}, \cite[Eq. 23]{taqqu} and properties of Laplace transform it follows easily that the $\gamma$ moment of $Y$ is 
\begin{equation}\label{gammasmomentofY}
U^\gamma(t) = \gamma\int_{0}^{t}U^{\gamma-1}(t-y)dU(y),
\end{equation}
for $\gamma \ge 1.$
\\

The next Lemma extends the result of Lager$\mathring{\mbox{a}}$s \cite{lageras} to any moment of order $\kappa > 0.$
\begin{lemma}\label{momentsofincrements}
Let $Y(t),\ t\geq 0$ be the inverse subordinator given by \eqref{inversesubordinator}. Then for all $s, t > 0$ and all real numbers $\kappa > 0,$
\begin{equation}\label{momentformula}
\mathbb{E}\left[|Y(t) - Y(s)|^{\kappa}\right] = U^{\kappa}(\mbox{max}(t,s)) - \kappa\int_{0}^{\mbox{min}(t,s)}U^{\kappa-1}(\mbox{max}(t,s)-y)\, dU(y),
\end{equation}
where $U(x) = \mathbb{E}[Y(x)]$.
\end{lemma}
\begin{proof}
Since $Y(t)$ is nondecreasing the case $\kappa = 1$ is trivial. In the remaining case, we consider $\kappa>0$ and $\kappa \neq 1.$ Write
\begin{equation}\label{momentsform}\mathbb{E}\left[|Y(t) - Y(s)|^\kappa\right] = \int_{0}^{\infty}\int_{0}^\infty |u-v|^\kappa H(du,dv),\end{equation}
a Lebesgue-Stieltjes integral with respect to the bivariate distribution function $H(u, v) :=\mathbb{P}[Y(t)\leq u, Y(s)\leq v]$ of the process $Y(t).$

To compute the integral in \eqref{momentsform}, we use the bivariate integration by parts formula \cite[Lemma 2.2]{gill}
\begin{eqnarray}
\int_0^a\int_0^b G(u,v)H(du, dv) &=& \int_0^a\int_0^b H([u, a]\times [v, b])G(du, dv)\nonumber\\ &+& \int_0^a H([u,a]\times(0,b])G(du,0)\nonumber\\ &+& \int_0^b H((0, a]\times [v,b])G(0, dv)\nonumber\\ &+& G(0,0)H((0, a]\times (0,b]),
\end{eqnarray}
with $G(u,v) = |u-v|^\kappa,$ and the limits of integration $a$ and $b$ are infinite:
\begin{eqnarray}
\int_0^\infty\int_0^\infty G(u,v)H(du, dv) &=& \int_0^\infty\int_0^\infty H([u, \infty]\times [v, \infty])G(du, dv)\nonumber\\ &+& \int_0^\infty H([u,\infty]\times(0,\infty])G(du,0)\nonumber\\ &+& \int_0^\infty H((0, \infty]\times [v,\infty])G(0, dv)\nonumber\\ &+& G(0,0)H((0, \infty]\times (0,\infty])\nonumber\end{eqnarray}
\begin{eqnarray}\label{corrfunction2}
&=&\int_0^\infty\int_0^\infty \mathbb{P}[Y(t)\geq u, Y(s)\geq v]G(du, dv)+\int_0^{\infty}\mathbb{P}[Y(t)\geq u]G(du,0)\nonumber\\ &+&\int_0^{\infty}\mathbb{P}[Y(s)\geq v]G(0,dv),\label{intbypartsresult}
\end{eqnarray}
since $Y(t)>0$ with probability $1$ for all $t>0.$ 
 Notice that $G(du,v) = g_v(u)\,du$ for all $v\geq 0,$ where 
\begin{equation}\label{gfunction}
g_v(u) = \kappa(u-v)^{\kappa - 1}I\{u > v\} + -\kappa (v-u)^{\kappa - 1}I\{u\leq v\}.
\end{equation}
In what follows we use notation for  the density of the inverse subordinator $Y(t)$ as $f_{Y(t)}(u)=f_t(u).$

Integrate by parts to get
\begin{eqnarray}
\int_0^\infty\mathbb{P}[Y(t)\geq u]G(du, 0)&=&\int_0^\infty (1-\mathbb{P}[Y(t)<u])\kappa\,u^{\kappa -1}\,du\nonumber\\
&=&\left[u^\kappa\mathbb{P}[Y(t)\geq u]\right]_0^\infty + \int_0^\infty u^\kappa f_t(u)\,du\nonumber\\
&=& \mathbb{E}[Y(t)^\kappa].
\end{eqnarray}
Similarly, $$\int_0^{\infty}\mathbb{P}[Y(s)\geq v]G(0, dv) = \int_0^\infty  v^{\kappa}f_s(v)\,dv = \mathbb{E}[Y(s)^\kappa],$$
and hence \eqref{intbypartsresult} reduces to
\begin{equation}\label{genform1}
\int_0^\infty\int_0^\infty G(u, v)H(du, dv) = I + \mathbb{E}[Y(t)^\kappa] +  \mathbb{E}[Y(s)^\kappa],
\end{equation}
where $$I = \int_0^\infty\int_0^\infty \mathbb{P}[Y_t\geq u, Y_s\geq v]G(du,dv).$$
Assume (without loss of generality) that $t\geq s$. Then $Y_t\geq Y_s$, so $\mathbb{P}[Y(t)\geq u, Y(s)\geq v] = \mathbb{P}[Y(s)\geq v]$ for $u\leq v.$ Write $I = I_1 + I_2 + I_3,$ where
\begin{eqnarray}
I_1:= \int_{u < v}\mathbb{P}[Y(t)\geq u, Y(s)\geq v]G(du,dv) = \int_{u<v}\mathbb{P}[Y(s)\geq v]G(du,dv)\nonumber\\ I_2:= \int_{u = v}\mathbb{P}[Y(t)\geq u, Y(s)\geq v]G(du,dv) = \int_{u=v}\mathbb{P}[Y(s)\geq v]G(du,dv)\nonumber\\ I_3:= \int_{u \geq v}\mathbb{P}[Y(t)\geq u, Y(s)\geq v]G(du,dv).\nonumber \ \  \ \ \ \ \ \  \ \ \ \ \ \  \ \ \ \ \ \ \ \ \ \ \ \ \ \ \ \  \ \ \ \  \ \nonumber
\end{eqnarray}
Since $G(du,dv) = -\kappa(\kappa-1)(v-u)^{\kappa -2}\,du\,dv$ for $u < v$, we may write
\begin{eqnarray}\label{I1part}
I_1 &=& -\kappa(\kappa-1)\int_{v=0}^\infty\int_{u=0}^v \mathbb{P}[E(s)\geq v](v-u)^{\kappa -2}\,du\,dv\nonumber\\ &=&-\kappa\int_{v=0}^\infty \mathbb{P}[E(s)\geq v]\, v^{\kappa - 1}\,dv\nonumber\\ &=& -\mathbb{E}[Y(s)^\kappa],
\end{eqnarray}
using the well-known formula $\mathbb{E}[X^\kappa] = \kappa\int_0^\infty x^{\kappa -1}\mathbb{P}[X\geq x]\,dx$ for any positive random variable.

 Since $G(du,v) = g_v(u)\,du$, where the function \eqref{gfunction} has no jump at the point $u=v$, we also have

\begin{equation}\label{I2part}
I_2 = \int_{u=v}\mathbb{P}[Y(s)\geq v]G(du,dv) = 0.
\end{equation}
Since $G(du,dv) = -\kappa(\kappa-1)(u-v)^{\kappa -2}\,du\ dv$ for $u>v$ as well, we have
\begin{eqnarray}\label{part3}
I_3 = -\kappa(\kappa-1)\int_{v=0}^\infty\mathbb{P}[Y(t)\geq u, Y(s)\geq v]\int_{u=v}^\infty (u-v)^{\kappa-2}\, du\, dv.
\end{eqnarray}
Next, we obtain an expression for $\mathbb{P}[Y(t)\geq u, Y(s)\geq v].$ Since the process $Y(t)$ is inverse to the stable subordinator $D(u),$ we have $\{Y(t)> u\}=\{D(u)<t\}$ \cite[Eq. (3.2)]{mark2}, and since $Y(t)$ has a density, it follows that $\mathbb{P}[Y(t)\geq u, Y(s)\geq v] = \mathbb{P}[D(u)<t, D(v)<s]$ (see \cite[Proposition A.2]{taqqu}). Let $g(x,u)$ be a density function of a random variable $D(u)$. Since $D(u)$ has stationary independent increments, it follows that
\begin{eqnarray}
\P[Y(t)\geq u, Y(s) \geq v]& =& \mathbb{P}[D(u)<t, D(v)<s]\nonumber\\
&=& \mathbb{P}[(D(u) - D(v)) + D(v) < t, D(v) <s]\nonumber\\
&=&\int_{y=0}^sg(y,v)\int_{x=0}^{t-y}g(x, u -v)\, dx\,dy,
\end{eqnarray}
substituting the above expression into \eqref{part3} and using the Fubini Theorem, it follow that
\begin{eqnarray}\label{part3estimate}
I_3&=& -\kappa(\kappa-1)\int_{y=0}^s\int_{x=0}^{t-y}\int_{v=0}^\infty g(y,v)\ dv\int_{u=v}^\infty g(x,u-v)(u-v)^{\kappa-2}du\ dx\ dy\nonumber\\
&=&-\kappa(\kappa-1)\int_{y=0}^s\int_{x=0}^{t-y}\int_{v=0}^\infty g(y,v)\ dv\int_{z=0}^\infty g(x,z)z^{\kappa-2}dz\,dx\,dy.
\end{eqnarray}
Let $h(y) = \int_{v=0}^\infty g(y,v)\ dv$ and $k(x) = \int_{z=0}^\infty g(x,z)z^{\kappa -2}\ dz.$ So, the Laplace transform of $h(y)$ is given by
\begin{eqnarray}
\mathcal{L}(h(y);\lambda) = \int_{y=0}^\infty e^{-\lambda y}h(y)\ dy &= & \int_{y=0}^\infty e^{-\lambda y}\int_{v=0}^\infty g(y,v)\,dv\,dy\nonumber\\
&=&\int_{y=0}^\infty \int_{v=0}^\infty e^{-\lambda y}g(y,v)\,dy\,dv\nonumber\\
&=&\int_{v=0}^\infty e^{-v\phi(\lambda)}\ dv\nonumber\\
&=&\frac{1}{\phi(\lambda)} = \mathcal{L}(U'(y);\lambda).\label{laplaceofh}
\end{eqnarray}
Similarly, take the Laplace transform of $k(x):$
\begin{eqnarray}
\mathcal{L}(k(x);\lambda) &=& \int_{x=0}^\infty e^{-\lambda x}\int_{z=0}^\infty  z^{\kappa-2}g(x,z)\,dz\,dx\nonumber\\
&=&\int_{z=0}^\infty z^{\kappa-2}\int_{x=0}^\infty e^{-\lambda x}g(x,z)\,dx\,dz\nonumber\\
&=&\int_{z=0}^\infty z^{\kappa-2} e^{-z\psi(\lambda)}dz\nonumber\\
& =& \frac{\Gamma(\kappa-1)}{(\phi(\lambda))^{\kappa-1}}. 
\end{eqnarray}
So, using properties of Laplace transform we get
\begin{equation}\label{estimateofkfunc}
\mathcal{L}\left((\kappa - 1)\int_{0}^{t}k(x)\, dx; \lambda\right) = \frac{\Gamma(\kappa)}{\lambda(\phi(\lambda))^{\kappa-1}} = \mathcal{L}\left(U^{\kappa-1}(t);\lambda\right).
\end{equation}
where we have used \eqref{laplacetranformofinverse}. Using the uniqueness theorem of the Laplace transform, \eqref{laplaceofh} and \eqref{estimateofkfunc}, we have
\begin{eqnarray}\label{I3part}
I_3 =
-\kappa\int_{y=0}^{s}U^{\kappa-1}(t-y)\ dU(y).
\end{eqnarray}
Now it follows using \eqref{genform1}, \eqref{I1part}, \eqref{I2part} and \eqref{I3part} that
\begin{eqnarray}
\mathbb{E}[|Y(t) - Y(s)|^\kappa]
&=&\int_{0}^{\infty}\int_{0}^{\infty}|u-v|^{\kappa}H(du, dv)\nonumber\\
&=&I_1 + I_2 + I_3 +  \mathbb{E}[Y(t)^\kappa] +  \mathbb{E}[Y(s)^\kappa]\nonumber\\
&=&-\mathbb{E}[Y(s)^\kappa]-\kappa\int_{y=0}^{s}U^{\kappa-1}(t-y)\ dU(y)\nonumber\\
&+& \mathbb{E}[Y(t)^\kappa] +  \mathbb{E}[Y(s)^\kappa]\nonumber\\
&=& \mathbb{E}[Y(\mbox{max}(t,s))^\kappa]-\kappa\int_{y=0}^{\mbox{min}(t,s)}U^{\kappa-1}(\mbox{max}(t,s)-y)\ dU(y),\label{momentsofincrement}\nonumber
\end{eqnarray}
which agrees with \eqref{momentformula}.
\end{proof}
\begin{remark}\label{covarianceofY}
Using the fact $xy = (x^2 + y^2 -(x-y)^2)/2,$ for $t\geq s > 0$ we get
\begin{equation}
\mathbb{E}[Y(t)Y(s)] = \frac{1}{2}U^2(s) + \int_0^sU(t-y)d\,U(y),
\end{equation}
and
\begin{eqnarray}
\mbox{Cov}[Y(t), Y(s)] &=& \mathbb{E}[Y(t)Y(s)]-U(t)U(s)\nonumber\\
&=& \frac{1}{2}U^2(s) + \int_{0}^{s}U(t-y)dU(y)- U(t)U(s),
\end{eqnarray}
which is equivalent to \cite[Corollary 4.3]{taqqu}.
\end{remark}
\begin{remark}
Let $Z(t) = B_H(Y(t))$. For any positive real number $m$ such that $mH > 0$ 
 using the facts that the inverse subordinator has non-decreasing paths and $B_H$ is $H$ self-similar with stationary increments. Conditioning on $Y$ we arrive at
\begin{eqnarray}\label{moementofincrementZ(t)}
\mathbb{E}\left[|Z(t) - Z(s)|^m\right] &=& \mathbb{E}\left[|B_H(1)|^m\right]\mathbb{E}\left[|Y(t) - Y(s)|^{mH}\right]\nonumber\\
&=& \frac{(2\sigma^2)^{m/2}}{\sqrt{\pi}}\Gamma\left(\frac{m+1}{2}\right) \mathbb{E}\left[|Y(t) - Y(s)|^{mH}\right]
\end{eqnarray}
since $B_H(1)$ has normal distribution with mean zero and standard deviation $\sigma$.
\end{remark}
\section{Correlation function}
In this section, we prove a general result that can be used to compute the correlation function of a time-changed fractional Brownian motion $Z(t) = B_H(Y(t))$ where $B_H, Y$ are independent, and $Y$ is a general inverse subordinator which is non-Markovian with non-stationary and non-independent increments.
With the above Lemma it is easy to obtain covariance function for time-changed processes $Z(t)$. The following theorem give the covariance function.
\begin{theorem}\label{maintheorem1}
Let $B_{H} = \{B_H(t), t\in \mathbb{R}\} $ be the fractional Brownian motion of index $H\in(0,1)$. Then the covariance function of the corresponding time-changed fractional Brownian motion $Z(t) = B_H(Y(t)),$ where $Y(t)$ is an independent inverse  subordinator \eqref{inversesubordinator} of $D(t)$ with Laplace exponent $\phi(\lambda)$ and $t,s\geq 0$, is given by
\begin{equation}\label{maintheoremstatement}
\mbox{Cov}\left(Z(t), Z(s)\right) = \frac{\sigma^2}{2}\left\{ U^{2H}(\mbox{min}(t,s))+ 2H\int_{0}^{\mbox{min}(t,s))}U^{2H-1}(\mbox{max}(t,s)-y)dU(y)\right\}
\end{equation}
where $\sigma^2 = \mbox{Var}(B_H(1)).$
\end{theorem}
\begin{proof}
Using the fact $B_H$ is $H$ self-similar with stationary increments for any $t, s\in \mathbb{R}$ we get
\begin{eqnarray}\mathbb{E}[B_H(t)B_H(s)] &=& \frac{1}{2}\left\{\mathbb{E}[B_H^2(t)] + \mathbb{E}[B_H^2(s)] - \mathbb{E}[(B_H(t)-B_H(s))^2]\right\}\nonumber\\
&=& \frac{1}{2}\left\{\mathbb{E}[B_H^2(t) + \mathbb{E}[B_H^2(s)] - \mathbb{E}[(B_H(t-s)-B_H(0))^2]\right\}\nonumber\\
&=&\frac{1}{2}\left\{|t|^{2H} + |s|^{2H} - |t - s|^{2H}\right\}\mathbb{E}[B_H^2(1)]\nonumber\\
&=&\frac{\sigma^2}{2}\left\{|t|^{2H} + |s|^{2H} - |t - s|^{2H}\right\}
\end{eqnarray}
since $B_H(0) = 0, B_H(1)\stackrel{d}{=}-B_H(-1)$ and $B_H(t)\stackrel{d}{=}|t|^HB_H(\mbox{sign}( t))$ for fixed t.

Now conditioning on $Y$ and using \eqref{momentsofincrement} we arrive at
\begin{eqnarray}
\mathbb{E}[Z(t)Z(s)] &=&  \frac{\sigma^2}{2}\left\{\mathbb{E}[|Y(t)|^{2H}] +\mathbb{E}[|Y(s)|^{2H}] - \mathbb{E}[|Y(t) - Y(s)|^{2H}]\right\}\nonumber\\
&=&\frac{\sigma^2}{2}\left\{U^{2H}(\mbox{min}(t,s)) + 2H\int_{0}^{\mbox{min}(t,s)}U^{2H-1}(\mbox{max}(t,s)-y)dU(y)\right\}.\nonumber
\end{eqnarray}
\end{proof}
\begin{remark}
For $t=s$ using a Laplace properties of a convolution function, the time-changed process $Z(t) = B_H(Y (t))$ has mean zero, its variance is
\begin{equation}\label{varianceofZ}
\mbox{Var}(Z(t)) = \sigma^2 U^{2H}(t),
\end{equation}
its correlation function is
\begin{equation}\mbox{corr}(Z(t), Z(s)) = \frac{\mbox{Cov}\left(Z(t), Z(s)\right)}{\sigma^2 \sqrt{U^{2H}(t)U^{2H}(s)}},\end{equation}
and when $H=1/2$ Brownian motion case the correlation function reduces to
\begin{equation}\mbox{corr}(Z(t),Z(s)) = \frac{U(\mbox{min}(t,s))}{\sqrt{U(t)U(s)}} = \sqrt{\frac{U(\mbox{min}(t,s))}{U(\mbox{max}(t,s))}},\end{equation}
as given in \cite[Remark 2.1]{leonenko2}.
\end{remark}
\begin{remark}\label{covarianceofdifference}
The covariance between far apart increments is given by
\begin{eqnarray}
&&\mbox{Cov}(Z(t)-Z(0), Z(t+v)-Z(v))\nonumber\\ &=& \mathbb{E}\left[Z(t)(Z(t+v)-Z(v))\right]\nonumber\\
&=&\mathbb{E}[Z(t)Z(t+v)] - \mathbb{E}[Z(t)Z(v)]\nonumber\\
&=&\frac{\sigma^2}{2}\left\{U^{2H}(t) + 2H\int_{0}^{t}U^{2H-1}(t+v-y)dU(y)\right\}\nonumber\\
&-&\frac{\sigma^2}{2}\left\{U^{2H}(t) + 2H\int_{0}^{t}U^{2H-1}(v-y)dU(y)\right\}\nonumber\\
&=& \sigma^2 H\int_0^t\left(U^{2H-1}(t+v-y) - U^{2H-1}(v-y)\right)dU(y)\nonumber,
\end{eqnarray}
for $v\ge t.$
\end{remark}
\section{Applications}\label{application}
In this section, we compute the correlation function for several examples. In view of Theorem \ref{maintheorem1}, the main technical issue is the computation of the function $U^k(t).$ For many inverse subordinators, the Laplace exponent $\phi$ can be written explicitly but the inversion to obtain $U^k(t)$ function may be difficult. Below we give examples from applications where the Laplace transform can be inverted analytically and where its asymptotic behavior can be found in order to characterize the behavior of the correlation function of time-changed fractional Brownian motion.
\begin{example}({\bf Inverse $\alpha-$stable subordinator}). Suppose $D(t)$ is standard $\alpha-$stable subordinator with index $0 < \alpha < 1$, so that the Laplace exponent $\phi(s) = s^{\alpha}$ for all $s > 0$. The inverse stable subordinator \eqref{inversesubordinator} has a Mittag-Leffler distribution:
$$\mathbb{E}\left[e^{-\lambda Y(t)}\right] = \sum_{k=0}^{\infty}\frac{(-\lambda t^\alpha)^k}{\Gamma(\alpha n + 1)} = E_\alpha(-\lambda t^\alpha),$$
where $E_\alpha$ is Mittag-Leffler function:
$$E_\alpha(z) = \sum_{k=0}^{\infty}\frac{z^k}{\Gamma(\alpha k + 1)},$$
Bingham \cite{bingham} and Bondesson et al. \cite{bondesson}.

Since $\widetilde{U}^{2H}(\lambda) = \Gamma(2H + 1)/\lambda^{2H\alpha + 1},$ then
\begin{equation}\label{renewalmoments}
U^{2H}(t) = \frac{\Gamma(2H+1)t^{2\alpha H}}{\Gamma(2\alpha H + 1)},
\end{equation}
and the variance of time-changed processes is
\begin{equation}\label{variancefun}
\mbox{Var}(Z(t))= \frac{\sigma^2\Gamma(2H+1)t^{2\alpha H}}{\Gamma(2\alpha H + 1)}.
\end{equation}
For $0< s\leq t,$ substitute \eqref{renewalmoments} into \eqref{maintheoremstatement} to see that the covariance function of
the time-changed processes is
\begin{eqnarray}
&&\mbox{Cov}(Z(t), Z(s))\\ &=& \frac{\sigma^2}{2}\left\{\frac{\Gamma(2H+1)s^{2\alpha H}}{\Gamma(2\alpha H + 1)}
 + \frac{2H\Gamma({2H})}{\Gamma(\alpha)\Gamma(\alpha(2H-1)+1)}\int_{0}^{s}(t-y)^{\alpha(2H-1)}y^{\alpha-1}\,dy\right\}\nonumber\\
&=&\frac{\sigma^2}{2}\left\{\frac{\Gamma(2H+1)s^{2\alpha H}}{\Gamma(2\alpha H + 1)} + \frac{\Gamma({2H+1})t^{2\alpha H}}{\Gamma(\alpha)\Gamma(\alpha(2H-1)+1)}\int_{0}^{s/t}(1-u)^{\alpha(2H-1)}u^{\alpha-1}\,du\right\}\nonumber\\
&=&\frac{\sigma^2}{2}\left\{\frac{\Gamma(2H+1)s^{2\alpha H}}{\Gamma(2\alpha H + 1)} + \frac{\Gamma({2H+1})t^{2\alpha H}}{\Gamma(\alpha)\Gamma(\alpha(2H-1)+1)}B(\alpha,\alpha(2H-1)+1;s/t)\right\},\nonumber
\end{eqnarray}
using a substitution $u = y/t$, where $B(a,b;z) := \int_{0}^z u^{a-1}(1-u)^{b-1}\,du$ is the
incomplete Beta function, and $B(a, b) := \Gamma(a)\Gamma(b)/\Gamma(a + b) = B(a, b; 1)$ is the Beta function. Apply the Taylor series expansion $(1-u)^{b-1} = 1 + O(u)$ as $u\rightarrow 0$ to see that
$$B(a, b;z) = \frac{z^a}{a} + O(z^{a + 1}),\ \ \ \mbox{as}\ \ z\rightarrow 0.$$
Then it follows that for $s > 0$ fixed and $t\rightarrow\infty$ we obtain
\begin{eqnarray}
G(\alpha, H;s, t)&:=& \frac{\Gamma({2H+1})t^{2\alpha H}}{\Gamma(\alpha)\Gamma(\alpha(2H-1)+1)}B(\alpha,\alpha(2H-1)+1;s/t)\nonumber\\
&=&  \frac{\Gamma({2H+1})t^{2\alpha H}}{\Gamma(\alpha +1)\Gamma(\alpha(2H-1)+1)}(s/t
)^\alpha + O((s/t)^{\alpha + 1}),\nonumber
\end{eqnarray}
so that
\begin{equation}\label{covariancefun}
\mbox{Cov}(Z(t), Z(s)) = \frac{\sigma^2}{2}\left\{\frac{\Gamma(2H+1)s^{2\alpha H}}{\Gamma(2\alpha H + 1)} + G(\alpha, H;s, t)\right\},
\end{equation}
where \begin{equation}\label{asymptotic}
G(\alpha, H;s, t)\rightarrow \frac{\Gamma({2H+1})t^{2\alpha H}}{\Gamma(\alpha +1)\Gamma(\alpha(2H-1)+1)}(s/t
)^\alpha,\ \ \ \mbox{as}\ \ t\rightarrow\infty.\end{equation}
Hence
\begin{equation}
\mbox{Cov}(Z(t), Z(s))\rightarrow\frac{\sigma^2}{2}\left\{\frac{\Gamma(2H+1)s^{2\alpha H}}{\Gamma(2\alpha H + 1)} + \frac{\Gamma({2H+1})t^{2\alpha H}}{\Gamma(\alpha +1)\Gamma(\alpha(2H-1)+1)}(s/t
)^\alpha\right\},\nonumber
\end{equation}
as $t\rightarrow \infty.$
From \eqref{variancefun} and \eqref{covariancefun} it follows that for $0<s\leq t$
$$\mbox{corr}(Z(t), Z(s)) = \frac{1}{2}\left\{\left(\frac{s}{t}\right)^{\alpha H} + \frac{\Gamma(2\alpha H + 1)}{\Gamma(2H + 1)}\frac{G(\alpha, H;s, t)}{(ts)^{\alpha H}}\right\},$$
and using \eqref{asymptotic} we have
\begin{equation}\label{longrange}
\mbox{corr}(Z(t), Z(s))\sim \frac{1}{2}\left\{\left(\frac{s}{t}\right)^{\alpha H} + \frac{1}{\alpha B(\alpha, \alpha(2H-1)+1)}\left(\frac{s}{t}\right)^{\alpha(1-H)}\right\},
\end{equation}
as $t\rightarrow\infty.$

When $H = 1/2$ for the special case when the outer processes $B_{1/2}(t)$ is a Brownian motion using  \eqref{variancefun} and \eqref{covariancefun} we get
$$\mbox{Cov}(Z(t), Z(s)) = \frac{\sigma^2}{\Gamma(\alpha+1)}s^\alpha,$$
and $$\mbox{corr}(Z(t),Z(s)) = \left(\frac{s}{t}\right)^{\alpha/2},$$
for $0<s\leq t,$ a formula obtained by Janczura and Wyloma$\acute{\mbox{n}}$ska \cite{Janczura}.

Similarly, by \eqref{momentformula} for $\kappa > 0$ and $0<s\leq t$ we have
$$\mathbb{E}\left[|Y(t) - Y(s)|^\kappa\right] = \frac{\Gamma(\kappa + 1)t^{\alpha\kappa}}{\Gamma(\alpha\kappa + 1)} - \frac{\Gamma(\kappa+1)t^{\alpha\kappa}}{\Gamma(\alpha)\Gamma(\alpha(\kappa-1)+1)}B(\alpha, \alpha(\kappa-1);s/t),$$
and by \eqref{moementofincrementZ(t)} for any positive real number $m$ such that $mH > 0$ we have
\begin{eqnarray}&&\mathbb{E}\left[|Z(t) - Z(s)|^m\right] \nonumber\\
&=& \frac{(2\sigma^2)^{m/2}\Gamma((m+1)/2)\Gamma(mH+1)t^{\alpha mH}}{\sqrt{\pi}}\bigg[\frac{1}{\Gamma(\alpha mH +1)}-\frac{B(\alpha, \alpha(mH-1);s/t)}{\Gamma(\alpha)\Gamma(\alpha(mH-1)+1)}\bigg].\nonumber\end{eqnarray}

Using remark \ref{covarianceofY} we can compute the covariance function of $Y$ for $t\ge s>0$:
\begin{equation}
\mathbb{E}[Y(t)Y(s)] = \frac{s^{2\alpha}}{\Gamma(2\alpha + 1)} + \frac{1}{\Gamma(\alpha)\Gamma(\alpha + 1)}\int_{0}^{s}(t-y)^{\alpha}y^{\alpha-1}dy,
\end{equation}
which implies $\partial_t\partial_s \mathbb{E}[Y(t)Y(s)] = 1/\Gamma^2(\alpha)[s(t-s)]^{1-\alpha}$ as given in Bingham \cite{bingham}.
\begin{eqnarray}
\mbox{Cov}(Y(t), Y(s)) &=& \frac{s^{2\alpha}}{\Gamma(2\alpha + 1)} + \frac{1}{\Gamma(\alpha)\Gamma(\alpha+1)}\int_{0}^{s}(t-y)^{\alpha}y^{\alpha-1}dy - \frac{(ts)^\alpha}{\Gamma(\alpha + 1)^2}\nonumber\\
&=& \frac{s^{2\alpha}}{\Gamma(2\alpha + 1)} + \frac{t^{2\alpha}}{\Gamma(\alpha)\Gamma(\alpha+1)}B(\alpha, \alpha+1;s/t) -\frac{(ts)^\alpha}{\Gamma(\alpha + 1)^2},
\end{eqnarray}
 which coincides with \cite[Eq. 9]{leonenko2}.
 
Finally, using remark \ref{covarianceofdifference} we show the covariance between far apart increments decrease to zeros as the power law. For fixed $t$ and as $v\rightarrow\infty$ we get
\begin{eqnarray}
&&\mbox{Cov}(Z(t) - Z(0), Z(t+v) - Z(v))\nonumber\\
&=&\frac{\sigma^2H\Gamma(2H)}{\Gamma(\alpha)\Gamma(\alpha(2H-1)+1)}\displaystyle\int_{0}^{t}\left[(t+v-y)^{\alpha(2H-1)}-(v-y)^{\alpha(2H-1)}\right]y^{\alpha-1}\, dy\nonumber\\
&=&\frac{\sigma^2H\Gamma(2H)v^{\alpha(2H-1)-1}}{\Gamma(\alpha)\Gamma(\alpha(2H-1)+1)}\displaystyle\int_{0}^{t}v\left[\left(1 + \frac{t-y}{v}\right)^{\alpha(2H-1)} - \left(1-\frac{y}{v}\right)^{\alpha(2H-1)}\right]y^{\alpha-1}\,dy\nonumber\\
&\sim&\frac{\sigma^2/2\Gamma(2H+1)\alpha(2H-1)v^{\alpha(2H-1)-1}}{\Gamma(\alpha)\Gamma(\alpha(2H-1) + 1)}\displaystyle\int_0^tty^{\alpha-1}\, dy\nonumber\\
&=&\frac{\sigma^2/2\Gamma(2H+1)t^{\alpha+1}}{\Gamma(\alpha + 1)\Gamma(\alpha(2H-1))}v^{\alpha(2H-1)-1}, \ \mbox{as}\ v\rightarrow\infty.
\end{eqnarray}

In summary, the correlation function of $Z(t)$ decays like a mixture of power law $t^{-\alpha H} + t^{-\alpha(1-H)}$. The non-stationarity time-changed process $Z(t)$ exhibits long range dependence (lack of summability of correlation). The covariance function between far apart increments of process $\mbox{Cov}[Z(t), Z(t+v)-Z(t)]$ decays like a power law $v^{-(1 + \alpha(1-2H))}$ which shows long-range dependence for $1/2 < H < 1$. Similar long range dependent behavior has been obtained for time-changed fractional Pearson diffusion \cite{leonenko2, jebessa2}.  
\end{example}
\begin{example}({\bf Inverse stable mixture}). Now consider a mixture of standard $\alpha-$stable subordinators with Laplace exponent
$$\Phi(\lambda) =  \int_{0}^{1}q(w)\lambda^w\,dw = \int_{0}^{\infty}(1 - e^{-\lambda x})l_q(x)\,dx,$$
where $q(w)$ is a probability density on $(0,1)$, and the density $l_q(x)$ of the L$\acute{\mbox{e}}$vy measure is given by
\begin{equation}
l_q(x) = \int_{0}^{1}\frac{wx^{-w-1}}{\Gamma(1-w)}q(w)\,dw.
\end{equation}
Such mixtures are used in time-fractional models of accelerating subdiffusion, see Mainardi et al. \cite{mainardi} and Chechkin et al. \cite{chechkin}. They can also be used to model ultraslow diffusion, see Sokolov et al. \cite{sokolov}, Meerschaert and Scheffler \cite{mark3}, and Kov$\acute{\mbox{a}}$cs and Meerschaert \cite{kovac}.

The $\alpha-$stable subordinator corresponds to the choice $q(w) = \delta(w-\alpha)$ where $\delta(\cdot)$ is the delta function. The model
$$q(w) = C_1\delta(w-\alpha_1) + C_2\delta(w-\alpha_2), \ \ C_1 + C_2 = 1,$$
with $\alpha_1 < \alpha_2$ was considered in Chechkin et al. \cite{chechkin}. The subordinator $D$ in this case is the linear combination of two independent stable subordinators with $\Phi(\lambda) = C_1\lambda^{\alpha_1} + C_2\lambda^{\alpha_2},$ so that
\begin{equation}\label{laplacetransformofMittag}
\widetilde{U}^k(\lambda) = \frac{\Gamma(k +1)}{\lambda(C_1\lambda^{\alpha_1} + C_2\lambda^{\alpha_2})^k} = \frac{\Gamma(k+1)\lambda^{-\alpha_2 k - 1}}{C_2^{k}\left(1 + \frac{C_1}{C_2}\lambda^{-(\alpha_2-\alpha_1)}\right)^k}.
\end{equation}
In order to invert analytically the Laplace transform \eqref{laplacetransformofMittag}, we use the well-known expression of the Laplace transform of the generalized Mittag-Leffler function (see \cite{saxena}, eq. 9), i.e.

\begin{equation}\label{laplaceofGMF}
\mathcal{L}(t^{\gamma-1}E^{\delta}_{\beta,\gamma}(\omega t^{\beta});\lambda)=\lambda^{-\gamma}\left(1-\omega \lambda^{-\beta}\right)^{-\delta},
\end{equation}
where $\mbox{Re}(\beta)>0, \mbox{Re}(\gamma)>0, \mbox{Re}(\delta)>0$ and $\lambda>|\omega|^{\frac{1}{Re(\beta)}}.$ The Generalized Mittag-Leffler (GML) function is defined as
\begin{equation}\label{GMLfunction1}
E^{\gamma}_{\alpha,\beta}(z) = \displaystyle\sum_{j=0}^{\infty}\frac{(\gamma)_jz^j}{j!\Gamma(\alpha j + \beta)},\ \ \ \alpha, \beta\in \mathbb{C}, Re(\alpha), Re(\beta), Re(\gamma)>0,
\end{equation}
where $(\gamma)_j=\gamma(\gamma + 1)\cdots (\gamma+j-1)$ (for $j=0,1,\ldots,\ \mbox{and}\ \gamma\neq 0$) is the Pochammer symbol and $(\gamma)_0=1.$ It is an entire function of order $\rho = [\mathcal{R}(\alpha)]^{-1}.$ When $\gamma = 1$ \eqref{GMLfunction1} reduces to the Mittag-Leffler function
\begin{equation}
E_{\alpha,\beta}(z) = \displaystyle\sum_{j=0}^{\infty}\frac{z^j}{\Gamma(\alpha j + \beta)}.
\end{equation}
Setting $\gamma = \alpha_2k + 1, \delta = k, \beta = \alpha_2 - \alpha_1$ and $\omega = -C_1/C_2$ we get
\begin{equation}\label{genrenwalfunformix}
U^k(t) = \frac{\Gamma(k + 1)}{C_2^k}t^{\alpha_2k}E^{k}_{\alpha_2-\alpha_1,\alpha_2k + 1}\left(-C_1t^{\alpha_2-\alpha_1}/C_2\right).
\end{equation}
For $k=1$ we have
$$\mathbb{E}[Y(t)] = U(t) = \frac{1}{C_2}t^{\alpha_2}E_{\alpha_2-\alpha_1,\alpha_2 + 1}\left(-C_1t^{\alpha_2-\alpha_1}/C_2\right).$$
Then \eqref{varianceofZ} implies that the time-changed process $Z(t) = B_H(Y(t))$ has variance
$$\mbox{Var}(Z(t)) = \frac{\sigma^2\Gamma(2H + 1)}{C_2^{2H}}t^{2\alpha_2H}E^{2H}_{\alpha_2-\alpha_1,2\alpha_2H + 1}\left(-C_1t^{\alpha_2-\alpha_1}/C_2\right).$$
We use the properties of generalized Mittag-Leffler function to obtain asymptotic expansion of the variance. From \cite[Eq. 2.59, p.15]{beghin} we obtain
\begin{equation}
\mbox{Var}(Z(t)) = \frac{\sigma^2\Gamma(2H+1)t^{2\alpha_1H}}{C_1^{2H}\Gamma(2\alpha_1H+1)} + o(t^{2\alpha_1H}),\ \ \ \mbox{as}\ \ t\rightarrow\infty.
\end{equation}
The asymptotic behavior for small $t$ can be deduced directly by the series expansion of \eqref{GMLfunction1}: indeed we get
\begin{equation}
\mbox{Var}(Z(t)) = \frac{\sigma^2\Gamma(2H+1)t^{2\alpha_2H}}{C_2^{2H}\Gamma(2\alpha_2H+1)}+ O\left(t^{\alpha_2(2H+1)-\alpha_1}\right),\ \ \mbox{as}\ \ t\rightarrow 0.
\end{equation}
For $0<s\leq t$ using \eqref{genrenwalfunformix} and \cite[Eq.2.59, p.15]{beghin} we have
\begin{eqnarray}
U^{2H-1}(t(1-sy/t))\simeq \frac{\Gamma(2H)t^{\alpha_1(2H-1)}(1-sz/t)^{\alpha_1(2H-1)}}{C_1^{2H-1}\Gamma(\alpha_1(2H-1)+1)},\ \ \mbox{as}\ \ t\rightarrow\infty\ \ \mbox{for}\ \ z\in[0,1].\nonumber
\end{eqnarray}
For a fixed $s>0$ and $t\rightarrow\infty$ we get
\begin{eqnarray}
\mbox{Cov}(Z(t),Z(s)) &=& \frac{\sigma^2}{2}\left\{U^{2H}(s) + 2Hs\int_{0}^{1}U^{2H-1}(t(1-sz/t))U^{'}(sz)dz\right\}\nonumber\\
&\sim&
\frac{\sigma^2}{2}\left\{U^{2H}(s) + \frac{\Gamma(2H+1)t^{\alpha_1(2H-1)}}{C_1^{2H-1}\Gamma(\alpha_1(2H-1)+1)}s\int_{0}^{1}U^{'}(sz)dz\right\}\nonumber\\
&=&\frac{\sigma^2}{2}\left\{U^{2H}(s) + \frac{\Gamma(2H+1)t^{\alpha_1(2H-1)}}{C_1^{2H-1}\Gamma(\alpha_1(2H-1)+1)}U(s)\right\}.
\end{eqnarray}
When $t$ is fixed and $s\rightarrow 0,$ then
\begin{equation}
\mbox{Cov}(Z(t), Z(s))\sim \frac{\sigma^2\Gamma(2H+1)s^{2\alpha_2H}}{2C_2^{2H}\Gamma(2\alpha_2H+1)}.
\end{equation}
\end{example}
\begin{example}{\bf (Inverse tempered stable subordinator).}

The standard tempered stable subordinator $D(t)$ with $0<\alpha<1$ is a L$\acute{\mbox{e}}vy$ process with tempered stable increments \cite{baeumer2, Rosinski}. The L$\acute{\mbox{e}}vy$ measure of the unit increment is
$$v(dx) = \frac{\alpha}{\Gamma(1-\alpha)}x^{-\alpha-1}e^{-\lambda}, x >0,$$
and then 
$$\mathbb{E}[e^{-\lambda D(t)}] = e^{-t\Phi(\lambda)} = \mbox{exp}\{-t((a + \lambda)^{\alpha} - a^{\alpha})\},$$
(see \cite[Section 7.2]{meerschaert-skorskii-book}).

Since $\widetilde{U}(\lambda) = 1/\lambda((a + \lambda)^\alpha - a^\alpha)$ using $$\frac{1}{(a + \lambda)^\alpha - a^\a} = \sum_{n=0}^{\infty}a^{\alpha n}(a + \lambda)^{-\alpha(1 + n)},$$
 we have
\begin{equation}\label{renewalfunctionoftempered}
\widetilde{U}(\lambda) = \sum_{n=0}^{\infty}\frac{a^{\alpha n}}{\lambda^{1 + \alpha(1 + n)}(1 + a\lambda^{-1})^{\alpha(1 + n)}}.
\end{equation}
Hence, using \eqref{laplaceofGMF} the renewal function
\begin{eqnarray}
U(t) = \mathbb{E}[Y(t)] &=& \sum_{n=0}^{\infty}a^{\alpha n}t^{\alpha(1 + n)}E^{\alpha(1+n)}_{1, \alpha(1 + n) + 1}(-at)\label{renewalfun}\\
&=& \sum_{n=0}^{\infty}a^{\alpha n}t^{\alpha(1 + n)}\frac{M(\alpha(1+n),\alpha(1 + n) + 1; -at)}{\Gamma(\alpha(1+n)+1)}\nonumber\\
&=& \sum_{n=0}^{\infty}a^{\alpha n}t^{\alpha(1 + n)}\frac{\alpha(1+n)(at)^{-\alpha(1+n)}\gamma(at;\alpha(1+n))}{\Gamma(\alpha(1+n)+1)}\nonumber\\
&=&\sum_{n=0}^{\infty}\frac{a^{-\alpha}\gamma(at;\alpha(1+n))}{\Gamma(\alpha(1 + n))}
=\sum_{n=0}^{\infty}a^{-\alpha}P(at,\alpha(1+n)),\nonumber
\end{eqnarray}
where $M(a,b;x) = \sum_{k=0}^{\infty}\frac{(a)_kx^k}{(b)_kk!}$ is Kummer's confluent hypergeometric function, $\gamma(x;v) = \int_{0}^{x}e^{-t}t^{v-1}\, dt$ is incomplete gamma function and $P(x/\theta,\beta)$ is the commutative distribution function for Gamma random variables with shape parameter $\beta$ and scale parameter $\theta$. We used the fact $M(a, a+1; -x) = ax^{-a}\gamma(x;a)$ in the above simplification \cite[see eq. 13.6.10]{Abramowitz}.

For $a=0$ in \eqref{renewalfun}, we have $$U(t) = \frac{t^\alpha}{\Gamma(1+\alpha)},$$
which is the renewal function for inverse $\alpha-$stable subordinator.

When $H=1/2,$ the Brownian motion case we get
\begin{equation}
\mbox{Var}(Z(t)) = \sigma^2U(t) = \sigma^2a^{-\alpha}\sum_{n=0}^{\infty}\frac{\gamma(at;\alpha(1+n))}{\Gamma(\alpha(1 + n))}.
\end{equation}
Since $\gamma(x;a) \sim x^a/a$ as $x\rightarrow 0,$ it follows:
\begin{eqnarray}
U(t) &\sim& t^{\alpha}\sum_{n=0}^{\infty}\frac{(at)^{\alpha n}}{\Gamma(1+\alpha(1+n))}\nonumber\\
&=& \frac{t^\alpha}{\Gamma(1+\alpha)} + O(t^{2\alpha}),\ \ \ \mbox{as}\ \ \ t\rightarrow 0,
\end{eqnarray}
and for $H=1/2$
\begin{equation}
\mbox{Var}(Z(t))\sim \frac{\sigma^2t^\alpha}{\Gamma(1+\alpha)} + O(t^{2\alpha}),\ \ \ \mbox{as}\ \ \ t\rightarrow 0.
\end{equation}

When $t$ is fixed and $s\rightarrow 0,$ then for $H=1/2$
\begin{equation}
\mbox{corr}[Z(t), Z(s)]=\sqrt{\frac{U(s)}{U(t)}}\sim \frac{s^{\alpha/2}}{\sqrt{\Gamma(1+\alpha)U(t)}} + O(s^\alpha).
\end{equation}
When $\lambda\rightarrow 0,$ the Laplace exponent $\Phi(\lambda) = (a + \lambda)^\a - a^\a\sim \a a^{\a -1}\lambda$ as $\lambda\rightarrow 0,$ and hence for $k>0$
$$\widetilde{U}^k(\lambda) = \frac{\Gamma(k + 1)}{\lambda\left((a + \lambda)^\a -a^\a\right)^k}\sim \frac{a^{(1-\a)k}}{\alpha^k}\Gamma(k+1)\lambda^{-1-k},\ \ \ \mbox{as}\ \ \lambda\rightarrow 0.$$
\end{example}
The Karamata Tauberian theorem \cite[Theorem 4, p.446]{feller} implies
\begin{equation}\label{gennthmomentofrenwal}U^k(t)\sim \frac{t^k}{\alpha^ka^{(\alpha-1)k}},\ \ \ \mbox{as}\ \ t\rightarrow\infty.\end{equation}
Hence the variance function of the process $Z(t)$ behaves as follows:
\begin{equation}
\mbox{Var}(Z(t)) = \sigma^2U^{2H}(t)\sim\frac{\sigma^2}{\alpha^{2H}a^{2H(\alpha-1)}}t^{2H}, \ \ \ \mbox{as}\ \ t\rightarrow\infty.
\end{equation}
For $0< s\leq t,$ using \eqref{gennthmomentofrenwal} and dominated convergence theorem we get
\begin{eqnarray}
\int_{0}^{s}U^{2H-1}(t-y)dU(y)&\sim&\frac{a^{(1-\alpha)(2H-1)}}{\alpha^{2H-1}}\int_{0}^{s}(t-y)^{2H-1}U'(y)dy\nonumber\\
&\sim&\frac{a^{(1-\alpha)(2H-1)}t^{2H-1}U(s)}{\alpha^{2H-1}},\ \ \mbox{as}\ \ t\rightarrow\infty.\nonumber
\end{eqnarray}
From \eqref{maintheoremstatement} it follows that for $0<s\leq t$
\begin{equation}
\mbox{Cov}(Z(t), Z(s))\sim \frac{\sigma^2}{2}\left\{U^{2H}(s) + \frac{2H a^{(1-\alpha)(2H-1)}t^{2H-1}U(s)}{\alpha^{2H-1}}\right\},\end{equation}
as $t\rightarrow\infty$, and hence
\begin{eqnarray}
\mbox{corr}(Z(t), Z(s)) &=& \frac{\mbox{Cov}(Z(t), Z(s))}{\mbox{Var}(Z(t)\mbox{Var}(Z(s))}\\
&\sim&\frac{1}{2}\left\{\alpha^{H}a^{(\alpha-1)H}\sqrt{U^{2H}(s)}\frac{1}{t^H} + \frac{2H\alpha^{1-H}a^{(\alpha-1)(1-H)}U(s)}{\sqrt{U^{2H}(s)}}\frac{1}{t^{1-H}}\right\}\nonumber
,\end{eqnarray}
as $t\rightarrow\infty.$

\section*{Acknowledgments} I would like to thank Erkan Nane for the careful reading of the paper and comments which considerably improved paper.


\begin{thebibliography}{99}
\bibitem{addison} P.S. Addison, Fractals and Chaos: An Illustrated course. Bristol, UK: Institute of Physics Publishing, 1997.
\bibitem{Applebaum}D. Applebaum, Levy Processes and Stochastic Calculus. Cambridge University Press, Cambridge, UK, 2004.
\bibitem{Abramowitz} M. Abramowitz, I. A. Stegun, editors {Handbook of Mathematical Functions,} National Bureau of Standards Applied Mathematics Series, \# 55, Washington, D.C.: U.S. Printing Office.
\bibitem{baeumer} B. Baeumer, M.M. Meerschaert, {Fractional diffusion with two time scales,} Physica A 373, 237-251, 2007.
\bibitem{baeumer2} B. Baeumer, M.M. Meerschaert, Tempered stable L$\acute{\mbox{e}}$vy motion and transient super-diffusion. Journal of Computational and Applied Mathematics 233, 2438-2448, 2010.
\bibitem{beghin} L.Beghin. {Random-time processes governed by differential equations of fractional distributed order}. Probab. Theory and Rel. Fields. {\bf 142} (2008), no. 3-4,  313-338.
\bibitem{bingham} N.H. Bingham, Limit theorems for occupation times of Markov processes, Z. Wahrscheinlichkeitstheor. Verwandte Geb. 17 1--22, 1971.
\bibitem{bondesson} L. Bondesson, G. Kristiansen, F. Steutel, Infinite divisibility of random variables and their integer parts, Statist. Probab. Lett. 28, 271--278, 1996.
\bibitem{caglar} M. Caglar, Fractional of fractional Brownian motion with micropulses. Advances in Performance Analysis, vol. 3, 43 - 49, 2000.
\bibitem{chechkin} A. V. Chechkin, R. Gorenflo and I. M. Sokolov, {Retarding subdiffusion and accelerating superdiffusion governed by distributed-order fractional diffusion equations.} Physical Review E 66, 046129, 2002.
\bibitem{feder} J. Feder, "Fractals" Physics of solids and liquids, 1988.
\bibitem{feller} W. Feller, {An introduction to probability theory and its applications}, Volume II.
\bibitem{gill} R.D. Gill, M.J. van der Laan, J.A. Wellner, {Inefficient estimators of the bivariate survival function for three models,} Ann. Inst. Henri Poincar$\acute{\mbox{e}}$ 31 (3), 545 - 597, 1995.
\bibitem{hurst} H. E. Hurst, {Long-term storage capacity in reservoirs.} Trans. Amer. Soc. Civil Eng. {\bf 116,} 400-410, 1951.
\bibitem{Janczura} J. Janczura, A. Wyloma$\acute{\mbox{n}}$ska, {Subdynamics of financial data from fractional Fokker-Planck equation}, Acta Physica Polonica B {\bf 40(5)}, 1341 - 1351, 2009.
\bibitem{kolmogorov-1} A. N. Kolmogorov, {Wienerische Spiralen und einige andere interessante Kurven im Hilbertschen Raum.} C. R. (Doklady) Acad. URSS (N.S.) {\bf 26}, 115-118, 1940.
\bibitem{kovac} M. Kov$\acute{\mbox{a}}$cs and M. M. Meerschaert, {Ultrafast subordinators and their hitting times.} Publications de l'Institut Mathematique, Nouvelle s$\acute{\mbox{e}}$rie, 80(94), 193-206; Memorial volume for Tatjana Ostrogorski, 2006.
\bibitem{lageras} A.N. Lager$\mathring{\mbox{a}}$s, {A renewal-process-type expression for the moments of inverse subordinators}, Journal of Applied Probability {\bf 42.4}, 1134-1144, 2005.
\bibitem{leonenko} N.N. Leonenko, M.M. Meerschaert, and A. Sikorskii, Correlation Structure of Fractional Pearson Diffusions, Computers and Mathematics with Applications, Vol. 66, No. 5, pp. 737ñ745, 2013.
\bibitem{leonenko2} N. N. Leonenko, M. M. Meerschaert, R. L. Schilling, A. Sikorskii, {Correlation structure of time-changed L$\acute{\mbox{e}}$vy processes}. Communications in Applied and Industrial Mathematics. ISSN 2038-0909.
\bibitem{magdziarz} Magdziarz, M., 2010. Path properties of subdiffusion – a martingale approach. Stoch. Models 26, 256-271.
\bibitem{mainardi} F. Mainardi, A. Mura, G. Pagnini, and R. Gorenflo, {Time-fractional diffusion of distributed order.} Journal of Vibration and Control 14, 1267 - 1290, 2008.
\bibitem{mandelbort1} B. B. Mandelbrot, The fractal geometry of natures. W. H. Freeman and Company, 1977.
\bibitem{mandelbrot} B.B. Mandelbrot, J.W. Van Ness, {Fractional Brownian motions, fractional noises and applications.} SIAM Review {\bf 10}, 422-437, 1968.
\bibitem{m-n-v-jmaa} M. M. Meerschaert,  E. Nane, and P. Vellaisamy, Distributed-order fractional diffusions on bounded domains. J. Math. Anal. Appl. {\bf 379} 216-228, 2011.
\bibitem{mark2} M. M. Meerschaert, H.P. Scheffler, Limit theorems for continuous time random walks with infinite mean waiting time, J. Appl. Probab. 41 623 - 638, 2004.
\bibitem{mark3} M. M. Meerschaert, H.P. Scheffler, {Stochastic model for ultraslow diffusion.} Stochastic Processes and Their Applications 116(9), 1215 - 1235, 2006.
\bibitem{mark} M. M. Meerschaert, H.P. Scheffler, Triangular array limits for continuous time random walks, Stochastic processes and their applications, 1606-1633, 2008.

\bibitem{meerschaert-skorskii-book} M.M. Meerschaert and A. Sikorskii, Stochastic Models for Fractional Calculus, De Gruyter Studies in Mathematics Vol. 43, 2012.
\bibitem{jebessa-nane-pams} J.B. Mijena and E. Nane. Strong analytic solutions of fractional Cauchy problems. Proceedings of the American Mathematical Society. 142, 1717-1731, 2014.
\bibitem{jebessa2} J. B. Mijena and  E. Nane .  Correlation structure of time-changed Pearson diffusions. Statistics \& Probability Letters. Volume {\bf 90}, 68-77, July 2014.
\bibitem{nane} E. Nane, Laws of the iterated logarithm for a class of iterated processes. Statist. Probab. Lett. {\bf 79}, 1744-1751, 2009.
\bibitem{podlubny} I. Podlubny, {Fractional differential equations}, Mathematics in Science and Engineering, Volume 198.
\bibitem{Rosinski} J. Rosi$\acute{\mbox{n}}$ski, Tempering stable processes. Stochastic Processes and Their Applications 17, 677-707, 2007.
\bibitem{samorodnitsk-taqqu} G. Samorodnitsky and M.S. Taqqu, {Stable Non-Gaussian processes: Stochastic models with infinite variance.} CHAPMAN \& HALL/CRC.
\bibitem{saxena} R.K. Saxena, A.M. Mathai, H.J. Haubold. {Reaction-diffusion  systems and nonlinear waves, Astrophysics and Space Science,} {\bf 305}, 297-303.
\bibitem{sokolov} I. M. Sokolov, A. V. Chechkin and J. Klafter, {Distributed order fractional kinetics.} Acta Physica Polonica B 35(4), 1323-1341, 2004.

\bibitem{taqqu} M. Veillete, M.S. Taqqu, {Using differential equations too obtain joint moments of first-passage times of increasing L$\acute{\mbox{e}}$ processes.} Statistics and Probability Letters 80, 697–705, 2010.














 



\end{thebibliography}
\end{document}